\documentclass[12pt]{amsart}
\usepackage{amsmath,amssymb,amsthm,amscd,amsfonts,geometry,color}
\theoremstyle{plain}

\newtheorem{nonumthm}{Theorem}

\newtheorem{nonumprop}{Proposition}
 
\newtheorem{nonumcor}{Corollary}

\theoremstyle{remark}

\theoremstyle{definition}

\newcommand{\cl}{\operatorname{cl}}               
\newcommand{\intr}{\operatorname{int}}            

\begin{document}

\title[On the Borel complexity and the complete metrizability of spaces of metrics]{On the Borel complexity and the complete metrizability of spaces of metrics}
\author{Katsuhisa Koshino}
\address[Katsuhisa Koshino]{Faculty of Engineering, Kanagawa University, Yokohama, 221-8686, Japan}
\email{ft160229no@kanagawa-u.ac.jp}
\maketitle

\begin{abstract}
Given a metrizable space $X$, let $AM(X)$ be the space of continuous bounded admissible metrics on $X$,
 which is endowed with the sup-metric.
In this paper, we shall investigate the Borel complexity and the complete metrizability of $AM(X)$ and show that a separable metrizable space $X$ is $\sigma$-compact if and only if $AM(X)$ is completely metrizable.
\end{abstract}

\section*{}

For a metrizable space $X$, let $C(X^2)$ be the space of continuous bounded real-valued functions on $X^2$, whose topology induced by the sup-metric $D$: for any $f, g \in C(X^2)$,
 $$D(f,g) = \sup\{|f(x,y) - g(x,y)| \mid (x,y) \in X^2\}.$$
Denote the subspace consisting of continuous bounded pseudometrics on $X$ by $PM(X)$ and the subspace consisting of bounded admissible metrics on $X$ by $AM(X)$, respectively.
In this paper, we shall study the Borel complexity and the complete metrizability of $AM(X)$.
Recall that the sup-metric $D$ is complete on $C(X^2)$ and $PM(X)$,
 which is closed in $C(X^2)$.
Let $AM^\ast(X)$ be the space of admissible metrics on $X$ with the topology induced by $D$ (namely, the topology of uniform convergence).
Y.~Ishiki \cite{Ish2} showed that $AM^\ast(X)$ is a Baire space,
 which implies that the open subspace $AM(X) \subset AM^\ast(X)$ is also Baire.
The author \cite{Kos20} proved that if $X$ is $\sigma$-compact,
 then $AM(X)$ is a $G_\delta$ set in $C(X^2)$,
 and hence it is completely metrizable, see also \cite{Ish1}.
It is unknown whether $AM(X)$ is always completely metrizable or not.
We shall show that even if $X$ is separable,
 the space $AM(X)$ is not necessarily completely metrizable.

Given a metrizable space $Y$, let $\mathfrak{A}_0(Y)$ be the family of open sets in $Y$, and let $\mathfrak{M}_0(Y)$ be the one of closed sets in $Y$.
For a natural number $n \geq 1$, define inductively the collections
 $$\mathfrak{A}_n(Y) = \Bigg\{\bigcup_{k \in \omega} A_k \ \Bigg| \ A_k \in \bigcup_{m < n} \mathfrak{M}_m(Y)\Bigg\} \text{ and } \mathfrak{M}_n(Y) = \Bigg\{\bigcap_{k \in \omega} A_k \ \Bigg| \ A_k \in \bigcup_{m < n} \mathfrak{A}_m(Y)\Bigg\}.$$
Denote by $\mathfrak{A}_n$ (respectively, $\mathfrak{M}_n$) a class of spaces $X$ satisfying that $X$ is belonging to $\mathfrak{A}_n(Y)$ (respectively, $\mathfrak{M}_n(Y)$) for every metrizable space $Y$ which contains $X$ as a subspace.
We call $\mathfrak{A}_n$ and $\mathfrak{M}_n$ to be the \textit{absolute Borel classes}.
Recall that $\mathfrak{A}_0 = \{\emptyset\}$ and $\mathfrak{M}_0$ is the class of compact metrizable spaces.
Moreover, the class $\mathfrak{M}_1$ is consisting of completely metrizable spaces,
 and the class $\mathfrak{A}_1$ is consisting of $\sigma$-locally compact metrizable spaces ($\sigma$-compact metrizable spaces in the separable case).
We can establish the following:

\begin{nonumthm}
Suppose that $X$ is a metrizable space and $n \geq 1$ is a natural number.
If $AM(X)$ is in $\mathfrak{A}_n$ (respectively, $\mathfrak{M}_n$),
 then $X$ is in $\mathfrak{M}_n$ (respectively, $\mathfrak{A}_n$).
\end{nonumthm}

As a corollary, we can give a necessary and sufficient condition on the complete metrizability of $AM(X)$ as follows:

\begin{nonumcor}
Let $X$ be a separable metrizable space.
Then $X$ is $\sigma$-compact if and only if $AM(X)$ is completely metrizable.
\end{nonumcor}

\begin{proof}
The ``only if'' part follows from \cite[Proposition~3]{Kos20} and the ``if'' part follows form the above theorem in the case where $n = 1$.
\end{proof}

Firstly, we give the following example,
 which contains a good technique to prove Theorem.

\begin{proof}[{\bf Example}]
Equipping the usual metric with the real line,
 let $Q = [0,1] \cap \mathbb{Q}$ and $P = [0,1] \cap \mathbb{P}$,
 where $\mathbb{Q}$ is the set of rationals and $\mathbb{P}$ is the set of irrationals.
Suppose that $X$ is the topological sum $P \oplus (0,1]$.
Then $AM(X)$ is not completely metrizable.
Since $Q$ is not completely metrizable,
 it is sufficient to show that $AM(X)$ contains a closed topological copy of $Q$.
For each $q \in Q$, define $d_q \in AM(X)$ by
 $$d_q(x,y) = \left\{
 \begin{array}{ll}
  |x - y| & \text{ if } (x,y) \in P^2 \text{ or } (0,1]^2,\\
  |x - q| + y & \text{ if } x \in P \text{ and } y \in (0,1],\\
  x + |y - q| & \text{ if } x \in (0,1] \text{ and } y \in P.
 \end{array}
 \right.$$
Put $S = \{d_q \in AM(X) \mid p \in Q\}$,
 so it is a closed subset of $AM(X)$ and is isometric to $Q$.

Firstly, we shall prove that $S$ is closed in $AM(X)$.
Take any sequence $\{d_{q_n}\} \subset S$ that is converging to some $d \in AM(X)$.
It remains to show that $d \in S$.
When $(x,y)$ is in $P^2$ or $(0,1]^2$,
 $$|d(x,y) - |x - y|| = |d(x,y) - d_{q_n}(x,y)| \leq D(d,d_{q_n}) \to 0$$
 as $n \to \infty$,
 which means that $d(x,y) = |x - y|$.
Replacing $\{q_n\}$ with a subsequence converging in $[0,1]$, we can find a point $z \in [0,1]$ so that $q_n \to z$ as $n$ tends to $\infty$.
For every $x \in P$ and every $y \in (0,1]$, since
 $$|d(x,y) - (|x - q_n| + y)| = |d(x,y) - d_{q_n}(x,y)| \leq D(d,d_{q_n}),$$
 $d(x,y) = |x - z| + y$.
In the case that $x \in (0,1]$ and $y \in P$, we have that $d(x,y) = x + |y - z|$ similarly.
Then $z \in Q$.
Otherwise, $z \in P$ and
 $$d\bigg(z,\frac{1}{n}\bigg) = |z - z| + \frac{1}{n} = \frac{1}{n} \to 0$$
 for the sequence $\{1/n\} \subset (0,1]$,
 that is not converging in $X$.
Therefore $d \notin AM(X)$,
 which is a contradiction.
It follows that $z \in Q$,
 and hence $d = d_z \in S$.

Next, we will show that $S$ is isometric to $Q$.
To prove it, let $i : S \to Q$ be a map defined by $i(d_q) = q$,
 and fix any $q, p \in Q$.
If $(x,y)$ is belonging to $P^2$ or $(0,1]^2$,
 then
 $$|d_q(x,y) - d_p(x,y)| = ||x - y| - |x - y|| = 0.$$
Moreover, if $x \in P$ and $y \in (0,1]$,
 then
 $$|d_q(x,y) - d_p(x,y)| = |(|x - q| + y) - (|x - p| + y)| \leq |q - p|.$$
Similarly, for each $x \in (0,1]$ and each $y \in P$, $|d_q(x,y) - d_p(x,y)| \leq |q - p|$.
It follows that $D(d_q,d_p) \leq |q - p|$.
We may assume that $q \leq p$.
In the case where $0 < q$, taking any $x \in [0,q] \cap P$ and any $y \in (0,1]$, we have that
 $$|d_q(x,y) - d_p(x,y)| = |(|x - q| + y) - (|x - p| + y)| = |q - p|.$$
In the case that $p < 1$, fix any $x \in [p,1] \cap P$ and any $y \in (0,1]$,
 so
 $$|d_q(x,y) - d_p(x,y)| = |(|x - q| + y) - (|x - p| + y)| = |q - p|.$$
When $q = 0$ and $p = 1$,
 for every number $\epsilon \in (0,1/2)$, letting any $x \in [0,\epsilon] \cap P$ and any $y \in (0,1]$, observe that
 $$|d_q(x,y) - d_p(x,y)| = |(|x - q| + y) - (|x - p| + y)| \geq |q - p| - 2\epsilon.$$
Consequently, $D(d_q,d_p) = |q - p|$,
 and hence the map $i$ is isometry.
It conclude that $AM(X)$ is not completely metrizable.
\end{proof}

In this example, the subset $P$ is closed in $X$ and is homeomorphic to $\mathbb{P}$.
More generally, by virtue of the similar argument, we have the following:

\begin{nonumprop}
If a metrizable space $X$ contains a closed subset homeomorphic to $\mathbb{P}$,
 then $AM(X)$ is not completely metrizable.
\end{nonumprop}

\begin{proof}
By the assumption, there exists a closed embedding $h : \mathbb{P} \to X$,
 so let $A_1 = h([0,1/3] \cap \mathbb{P})$ and $A_2 = h([2/3,1] \cap \mathbb{P})$.
Set $Q' = [2/3,1] \cap \mathbb{Q}$ and define a function $i : Q' \to AM(A_1 \oplus A_2)$ by for every $q \in Q'$,
 $$i(q)(x,y) = \left\{
 \begin{array}{ll}
  |h^{-1}(x) - h^{-1}y| & \text{ if } (x,y) \in A_1^2 \text{ or } A_2^2,\\
  h^{-1}(x) + |h^{-1}(y) - q| & \text{ if } x \in A_1 \text{ and } y \in A_2,\\
  |h^{-1}(x) - q| + h^{-1}(y) & \text{ if } x \in A_2 \text{ and } y \in A_1.
 \end{array}
 \right.$$
It follows form the same argument as the above example that $i$ is an isometric embedding.

According to \cite[Theorem~2]{BB} (see also \cite{Bess, Ba1, Pi, Z}), we can obtain a continuous map $e : AM(A_1 \oplus A_2) \to AM(X)$ such that $e(d)|_{(A_1 \oplus A_2)^2} = d$ for every $d \in AM(A_1 \oplus A_2)$.
Then the composition $e \circ i$ is a closed embedding.
For any $q, p \in Q'$ with $q < p$ and any $x \in A_1$, letting $y = h(r)$ with $r \in [q,(2q + p)/3] \cap \mathbb{P}$, observe that
\begin{align*}
 |e \circ i(q)(x,y) - e \circ i(p)(x,y)| &= |i(q)(x,y) - i(p)(x,y)|\\
 &= |(h^{-1}(x) + |h^{-1}(y) - q|) - (h^{-1}(x) + |h^{-1}(y) - p|)|\\
 &> (p - q)/3 > 0,
\end{align*}
 which implies that $e \circ i$ is injective.
It remains to prove that $e \circ i$ is a closed map.
Suppose that $\{q_n\}$ is a sequence in $Q'$ such that $e \circ i(q_n)$ is converging to some metric $d \in AM(X)$.
Then there exists a subsequence of $\{q_n\}$ that converges to some point $z \in Q'$.
Indeed, replace $\{q_n\}$ with a converging subsequence to some $z \in [2/3,1]$.
By the similar way as Example, for every $x \in A_1$ and $y \in A_2$, $d(x,y) = h^{-1}(x) + |h^{-1}(y) - z|$.
Assume that $z \in \mathbb{P}$.
Take any sequence $\{p_n\}$ consisting irrational numbers so that $1/3 \geq p_n \to 0$ as $n \to \infty$,
 so
 $$d(h(z),h(p_n)) = p_n + |z - z| = p_n \to 0.$$
On the other hand, the sequence $\{h(p_n)\}$ is not converging in $A_1 \oplus A_2$,
 which is a contradiction.
Hence $z \in Q'$ and $e \circ i$ is a closed embedding.
Since $AM(X)$ contains a closed subset homeomorphic to $Q'$,
 which is not completely metrizable,
 $AM(X)$ is also not completely metrizable.
\end{proof}

Every space homeomorphic to some member of an absolute Borel class also belongs to it,
 and all closed subspaces of some space in an absolute Borel class also belong to it.
According to \cite[4.5.8~(a)]{E}, if a space is a union of a locally finite family consisting of subspaces belonging to an absolute Borel class,
 then it is also in the same class.
It is known that for a metrizable space $X$, $X \in \mathfrak{A}_n$, $n \geq 2$, (respectively, $X \in \mathfrak{M}_n$, $n \geq 1$,) if and only if $X \in \mathfrak{A}_n(Y)$ (respectively, $\mathfrak{M}_n(Y)$) for some completely metrizable space $Y$, refer to \cite[Theorem~5.11.2]{Sakaik12}.
Observe that if $X \notin \mathfrak{A}_1$,
 then there exists a completely metrizable space $Y$ such that $X \notin \mathfrak{A}_1(Y)$.
For spaces $A \subset Y$, denote the interior and closure of $A$ in $Y$ by $\intr_Y{A}$ and $\cl_Y{A}$, respectively.
Now we will show Theorem.

\begin{proof}[Proof of Theorem]
We only prove it in the case that $AM(X) \in \mathfrak{M}_n$.
Assume that $X$ is not belonging to $\mathfrak{A}_n$ and take a complete metric space $Y = (Y,\rho)$ such that $X \notin \mathfrak{A}_n(Y)$.
Then the subset $B = X \setminus \intr_Y{X} \notin \mathfrak{A}_n(Y)$,
 and hence the subspace $Z = \cl_Y{(X \setminus \intr_Y{X})} \setminus X \notin \mathfrak{M}_n$.
Note that $\cl_Y{(X \setminus \intr_Y{X})} = B \cup Z$ is complete and it has a locally finite cover $\mathcal{C}$ consisting of closed subsets of $Y$ satisfying the following:
 \begin{enumerate}
  \item there is $C \in \mathcal{C}$ such that $Z \cap C \notin \mathfrak{M}_n$;
  \item there exists $C' \in \mathcal{C}$ such that $Z \cap C' \neq \emptyset$ but $C \cap C' = \emptyset$.
 \end{enumerate}
Fix any point $a \in Z \cap C'$ and find a sequence $\{a_n\} \subset X \cap C'$ converging to the point $a$.
Let $A = \{a_n\}$, $B' = B \cap C$ and $Z' = Z \cap C$.
Remark that $A$ and $B'$ are closed in $X$.
Define a continuous function $i : Z' \to AM(A \oplus B')$ by
 $$i(z)(x,y) = \left\{
 \begin{array}{ll}
  \rho(x,y) & \text{ if } (x,y) \in A^2 \text{ or } B'^2,\\
  \rho(x,a) + \rho(y,z) & \text{ if } x \in A \text{ and } y \in B',\\
  \rho(x,z) + \rho(y,a) & \text{ if } x \in B' \text{ and } y \in A.
 \end{array}
 \right.$$

Choosing an extending map $e : AM(A \oplus B') \to AM(X)$ as in Proposition, we can get an closed embedding $e \circ i$.
Indeed, let any distinct points $z_1, z_2 \in Z'$.
Since $B$ is dense in $\cl_Y{(X \setminus \intr_Y{X})}$,
 find a point $x \in B'$ with $\rho(x,z_1) \leq \rho(z_1,z_2)/3$.
For every $y \in A$,
\begin{align*}
 |e \circ i(z_1)(x,y) - e \circ i(z_2)(x,y)| &= |i(z_1)(x,y) - i(z_2)(x,y)|\\
 &= |(\rho(x,z_1) + \rho(y,a)) - (\rho(x,z_2) + \rho(y,a))|\\
 &\geq \rho(z_1,z_2)/3 > 0,
\end{align*}
 and hence $e \circ i$ is an injection.
To prove that $e \circ i$ is a closed map, take any sequence $\{z_n\}$ in $Z'$ so that $e \circ i(z_n)$ is converging to some $d \in AM(X)$.
As is observed in the above inequality,
 for all natural numbers $n$ and $m$, $\rho(z_n,z_m) \leq 3D(e \circ i(z_n),e \circ i(z_m))$.
Thus $\{z_n\}$ is a Cauchy sequence in the complete metric space $\cl_Y(X \setminus \intr_Y{X}) \cap C$,
 so it is converging to some $z \in \cl_Y(X \setminus \intr_Y{X}) \cap C$.
Supposing that $z \in B'$,
 we have that
 $$d(z,a_n) = \rho(z,z) + \rho(a_n,a) \to 0,$$
 but $\{a_n\}$ is not converging in $A \oplus B'$.
This is a contradiction.
Hence the point $z \in Z'$ and the composition $e \circ i$ is a closed embedding.
Therefore $AM(X)$ admits a closed embedding from $Z'$,
 that is not belonging to $\mathfrak{M}_n$.
Consequently, the space $AM(X) \notin \mathfrak{M}_n$.
The proof is completed.
\end{proof}

\begin{proof}[{\bf Remark}]
Theorem holds on $AM^\ast(X)$ because $AM(X)$ is open in $AM^\ast(X)$ and any open subspace of a member in $\mathfrak{A}_n$ (respectively, $\mathfrak{M}_n$), $n \geq 1$, is also belonging to $\mathfrak{A}_n$ (respectively, $\mathfrak{M}_n$).
\end{proof}

\subsection*{Acknowledgements}

This study is motivated by some works of Yoshito Ishiki,
 and the author would like to thank him for his helpful advice on the non-separable case of Theorem and on the unbounded case in Remark.

\end{document}